\documentclass[11pt]{article}
\usepackage{amsmath, amsthm, epsfig,amssymb, multicol,xcolor,tikz}
\usepackage[active]{srcltx}
\usepackage{fullpage}
\usepackage{hyperref}

\newcommand{\floor}[1]{{\lfloor #1 \rfloor}}
\newcommand{\bfloor}[1]{{\left\lfloor #1 \right\rfloor}}
\newcommand{\ceil}[1]{{\lceil #1 \rceil}}

\newcommand{\kgmcomment}[1]{}

\newtheorem{theorem}{Theorem}[section]
 
\newtheorem*{thm*}{Theorem} 
\newtheorem{lemma}[theorem]{Lemma}

\newtheorem{corollary}[theorem]{Corollary}

\newtheorem*{cor*}{Corollary}

\theoremstyle{definition}

\newtheorem{definition}[theorem]{Definition}

\newtheorem{conjecture}[theorem]{Conjecture}

\newcommand{\st}{\colon\,}

\title{Monotone Paths in Dense Edge-Ordered Graphs}
\author{Kevin G. Milans
\thanks{West Virginia University, Morgantown, WV 26505,
({\tt milans@math.wvu.edu}). } }

\begin{document}
\maketitle

\newcommand{\el}{\ell}
\newcommand{\ml}{f}
\newcommand{\mt}{f^{\star}}
\begin{abstract}
The \emph{altitude} of a graph $G$, denoted $\ml(G)$, is the largest integer $k$ such that under each ordering of $E(G)$, there exists a path of length $k$ which traverses edges in increasing order.  In 1971, Chv\'atal and Koml\'os asked for $\ml(K_n)$, where $K_n$ is the complete graph on $n$ vertices.  In 1973, Graham and Kleitman proved that $\ml(K_n) \ge \sqrt{n - 3/4} - 1/2$ and in 1984, Calderbank, Chung, and Sturtevant proved that $\ml(K_n) \le (\frac{1}{2} + o(1))n$.  We show that $\ml(K_n) \ge (\frac{1}{20} - o(1))(n/\lg n)^{2/3}$.
\end{abstract}

\section{Introduction}

A \emph{totally ordered graph} is a graph $G$ that is associated with a total ordering of its vertex set $V(G)$ and a total ordering of its edge set $E(G)$.  We use $T(G)$ and $T'(G)$ to denote the total orderings of $V(G)$ and $E(G)$ respectively.  When only the vertices or only the edges of $G$ are totally ordered, we call $G$ an \emph{ordered graph} or an \emph{edge-ordered graph}, respectively.  An \emph{ordering}, \emph{edge-ordering}, or \emph{total ordering} of a graph $G$ is an ordered, edge-ordered, or totally ordered graph whose underlying graph is $G$.

In an edge-ordered graph $G$, a \emph{monotone path} is a path which traverses edges in increasing order with respect to $T'(G)$.  A \emph{monotone trail} is similar, except that a trail is allowed to revisit vertices.  The \emph{altitude} of a graph $G$, denoted $\ml(G)$, is the maximum integer $k$ such that every edge-ordering of $G$ contains a monotone path of length $k$.  Also, let $\mt(G)$ be the maximum integer $k$ such that every edge-ordering of $G$ contains a monotone trail of length $k$.  

In 1971, Chv\'atal and Koml\'os~\cite{CK} asked for $\ml(K_n)$ and $\mt(K_n)$, where $K_n$ denotes the complete graph on $n$ vertices.  Citing private communication, Chv\'atal and Koml\'os noted in their 1971 paper that Graham and Kleitman had already proved $\Omega(n^{1/2}) \le \ml(K_n) < (\frac{3}{4} + \varepsilon) n$ and established $\mt(K_n)$ exactly:  $\mt(K_n) = n-1$ unless $n\in\{3,5\}$, in which case $\mt(K_n) = n$.

To show $\mt(K_n) \ge n-1$, Graham and Kleitman~\cite{GK} proved that if $G$ has average degree $d$, then $\mt(G) \ge d$.  Friedgut communicated to Winkler~\cite{W} an elegant formulation of their proof, known as the \emph{pedestrian argument}.  For an $n$-vertex edge-ordered graph $G$, the pedestrian argument involves $n$ pedestrians, with one starting at each vertex in $G$.  An announcer calls out the names of the edges in order according to $T'(G)$.  When $e$ is called, both pedestrians at the endpoints of $e$ traverse $e$, trading places.  Since each pedestrian travels along a monotone trail and each edge is traversed by two pedestrians, the average length of a pedestrian's monotone trail is $2|E(G)|/n$, which equals $d$.  The pedestrian argument has recently been modified to produce monotone paths (see~\cite{LL} and~\cite{DMPRT}).

Determining the altitude of a graph appears to be difficult in general.  In 1973, Graham and Kleitman~\cite{GK} published their results on $\ml(K_n)$ and $\mt(G)$.  In particular, they proved that $\sqrt{n-3/4} - 1/2 \le \ml(K_n) < 3n/4$, and they conjectured that $\ml(K_n)$ is closer to their upper bound than their lower bound.  They also commented that, with additional effort, their lower bound could be improved to $\ml(K_n) \ge (c-o(1))\sqrt{n}$ for some $c>1$.  In his Master's thesis from the same year, R\"odl~\cite{R} proved that if $G$ has average degree $d$, then $\ml(G)\ge(1-o(1))\sqrt{d}$; for $G=K_n$, R\"odl's result matches the Graham--Kleitman lower bound asymptotically.  R\"odl also noticed that the ideas in the Graham--Kleitman upper bound can be combined with results in design theory to prove $\ml(K_n) \le (\frac{2}{3} + o(1))n$.  Alspach, Heinrich, and Graham (unpublished, see~\cite{CCS}) further improved the upper bound to $\ml(G) \le (\frac{7}{12}+o(1))n$.  In 1984, Calderbank, Chung, and Sturtevant~\cite{CCS} obtained the best known upper bound: $\ml(K_n) \le (\frac{1}{2} + o(1))n$.  After 1984, explicit progress on determining $\ml(K_n)$ slowed (but see~\cite{BCM} for exact values for $n \le 8$).  In the meantime, other interesting results on the altitude of graphs have appeared.  

In 2001, Roditty, Shoham, and Yuster~\cite{RSY} proved that $\ml(G) \le 9$ if $G$ is planar and showed that $\ml(C_n \vee \overline{K_2}) \ge 5$ for $n\ge 99$, where $C_n \vee \overline{K_2}$ is the planar graph obtained by joining the $n$-vertex cycle $C_n$ and a pair of non-adjacent vertices.  Consequently, the maximum of altitude of a planar graph is between $5$ and $9$.

Clearly, $\ml(G) \le \mt(G)$.  The \emph{edge-chromatic number} of $G$, denoted $\chi'(G)$, is the minimum $k$ such that $E(G)$ is the union of $k$ matchings.  Ordering $E(G)$ so that each matching is an interval shows that $\mt(G) \le \chi'(G)$.  Vizing's theorem~\cite{V} states that $\chi'(G) \le \Delta(G) + 1$, where $\Delta(G)$ is the maximum degree of $G$.  It follows that $\ml(G) \le \Delta(G) + 1$.  

Improving a result of Yuster~\cite{Y}, Alon~\cite{A} gave a short proof that there exist $k$-regular graphs $G$ with $\ml(G) \ge k$, as follows.  The \emph{girth} of $G$ is the length of a shortest cycle in $G$.  If $G$ has girth $g$, then every trail of length less than $g$ is a path.  Therefore $\ml(G) \ge \min\{g-1,\mt(G)\} \ge \min\{g-1,d\}$, where $d$ is the average degree of $G$.  In particular, if $G$ is $k$-regular and has girth larger than $k$, then $\ml(G) \ge k$.  For $k=3$, better constructions are known.  Mynhardt, Burger, Clark, Falvai1, and Henderson~\cite{MBCFH} characterized the $3$-regular graphs with girth at least $5$ and altitude 3, and then used the characterization to show that the flower snarks are examples of $3$-regular graphs with altitude $4$.  For $k \ge 4$, it remains open to decide whether there are graphs $G$ with $\Delta(G) = k$ and $\ml(G) = k+1$.  

A \emph{Hamiltonian path} in a graph is a path containing all of its vertices.  Katreni{\v{c}} and and Semani{\v{s}}in~\cite{KS} proved that deciding whether a given edge-ordered graph contains a Hamiltonian monotone path is NP-complete.  Although it seems likely that computing the altitude of a given graph is NP-hard or worse, we note that the result of Katreni\v{c} and Semani\v{s}in does not directly imply this.  

Lavrov and Loh~\cite{LL} investigated the maximum length of a monotone path in a random edge-ordering of $K_n$.  They showed that with probability tending to $1$, a random edge-ordering of $K_n$ contains a monotone path of length at least $0.85n$.  Consequently, edge-orderings of $K_n$ that give sublinear upper bounds on $\ml(K_n)$, if they exist, are rare.  They also proved that with probability at least $1/e - o(1)$, a random edge-ordering of $K_n$ contains a Hamiltonian monotone path.  The common strengthening of these results leads to a natural and beautiful conjecture. 

\begin{conjecture}[Lavrov--Loh~\cite{LL}]
With probability tending to $1$, a random edge-ordering of $K_n$ contains a Hamiltonian monotone path.
\end{conjecture}

Recently, De Silva, Molla, Pfender, Retter, and Tait~\cite{DMPRT} proved that $\ml(Q_n) \ge n/\lg n$ where $Q_n$ is the $n$-dimensional hypercube and $\lg$ denotes the base-2 logarithm.  They also showed that if $\omega(n) \to \infty$ and $p \le (\omega(n)\ln n)/n^{1/2}$, then with probability tending to $1$ the Erd\H{o}s--R\'enyi random graph $G(n,p)$ has altitude at least $(1-o(1))\frac{np}{\omega(n)\ln n}$.  Consequently, there are graphs with average degree $\sqrt{n}(\ln n)^2$ and altitude at least $(1-o(1))\sqrt{n}$.  These graphs are sparse and yet the lower bound on their altitude asymptotically matches the lower bound on $\ml(K_n)$ due to Graham and Kleitman.

In this paper, we improve R\"odl's result for sufficiently dense graphs.  We show that if $G$ is an $n$-vertex graph with average degree $d$ and $s^2/d \to 0$ where $s = \Theta(n^{1/3}(\log n)^{2/3})$, then $\ml(G) \ge (1-o(1))\frac{d}{4s}$.  For $G=K_n$, we obtain $\ml(K_n) \ge (\frac{1}{20} - o(1))(n/\lg n)^{2/3}$.  Our proof is based on a simple algorithm to extend monotone paths.

\section{Monotone Path Algorithm}\label{sec:alg}

In his Master's thesis, R\"odl~\cite{R} gave a simple and elegant argument that $\ml(G)\ge (1-o(1))\sqrt{d}$ where $d$ is the average degree of $G$, which we outline as follows.  Let $G$ be an edge-ordered graph with average degree $d$ and suppose that $k$ is an integer with $d\ge 2\binom{k+1}{2} = 2(1 + \cdots + k)$.  Obtain $G'$ from $G$ by marking at each vertex $v$ the $k$ largest edges incident to $v$ (or all edges incident to $v$ if $d(v)<k$) and then removing all marked edges.  Since $G'$ has average degree at least $d-2k$, by induction $G'$ contains a monotone path $x_0\ldots x_{k-1}$ of length $k-1$.  Since $x_{k-1}$ is not isolated in $G'$, it follows that $x_{k-1}$ is incident to at least $k$ edges in $E(G)-E(G')$, and one of these extends $x_0 \ldots x_{k-1}$ to a monotone path of length $k$.  R\"odl's idea of reserving large edges at each vertex for path extension plays a key role in our approach.  We make a slight change in that we require the vertices to have disjoint sets of reserved edges.  We organize the edges in a table.

Let $G$ be a totally ordered graph.  The \emph{height table} of $G$ is an array $A$ whose columns are indexed by $V(G)$ and rows are indexed by the positive integers.  Each cell in $A$ is empty or contains an edge in $G$.  For $u\in V(G)$ and a positive integer $i$, we use $A(i,u)$ to denote the contents of the cell in $A$ located in row $i$ and column $u$.  We order the cells of $A$ so that $A(i,u)$ precedes $A(i',u')$ if and only if $i<i'$ or $i=i'$ and $u$ precedes $u'$ in $T(G)$.  We define $A$ iteratively.  Given that the contents of all preceding cells have been defined, let $A(i,u)$ be the largest edge (relative to $T'(G)$) incident to $u$ not appearing in a preceding cell; if no such edge exists, then $A(i,u)$ is empty.  Note that each edge appears in exactly one cell in $A$.  We define the \emph{height} of $e$ in $G$, denoted $h_G(e)$, to be the index of the row in $A$ containing $e$.

Extending a given monotone path is a key step in our algorithm.
The \emph{height} of a nontrivial monotone path $x_0 \ldots x_k$ is the height of its last edge $x_{k-1}x_k$.

\begin{lemma}\label{lem:ext}
Let $G$ be a totally ordered graph.  For $1 \le k < r$, each monotone path of length $k$ and height $r$ extends to a monotone path of length $k+1$ and height at least $r-k$.
\end{lemma}
\begin{proof}
Let $A$ be the height table of $G$, and let $x_0 \ldots x_k$ be a monotone path of length $k$ and height $r$.  Since $x_{k-1}x_k$ appears in row $r$ in $A$, this edge did not already appear when $A(i,x_k)$ is defined for $i<r$.  It follows that for $i<r$, the cell $A(i,x_k)$ contains an edge incident to $x_k$ which is larger than $x_{k-1}x_k$ in $T'(G)$.  Let $S = \{A(i,x_k)\st r-k\le i\le r-1\}$.  Since $|S| = k$ and $x_{k-1}x_k\not \in S$, some edge in $S$ joins $x_k$ with a vertex outside $\{x_0, \ldots, x_{k-1}\}$ and extends the path as claimed.
\end{proof}

Starting with a single edge and iterating Lemma~\ref{lem:ext}, we obtain the following.

\begin{lemma}\label{lem:extlong}
Let $G$ be a totally ordered graph and let $x_0x_1$ be an edge in $G$ of height $r$.  If $t$ is a positive integer and $\binom{t}{2}<r$, then $G$ contains monotone path $x_0x_1\ldots x_t$ of height at least $r-\binom{t}{2}$.
\end{lemma}
\begin{proof}
By induction on $t$.  The lemma is clear when $t=1$.  For $t>1$, the inductive hypothesis implies that $G$ contains a monotone path $x_0x_1\ldots x_{t-1}$ of height at least $r-\binom{t-1}{2}$.  With $k=t-1$, we apply Lemma~\ref{lem:ext} to obtain a monotone path $x_0\ldots x_t$ with height at least $(r-\binom{t-1}{2}) - (t-1)$ which equals $r-\binom{t}{2}$.
\end{proof}

Using Lemma~\ref{lem:extlong}, we match R\"odl's bound $\ml(G)\ge (1-o(1))\sqrt{d}$ asymptotically.  We include the short proof for completeness.

\begin{theorem}
If $G$ has average degree $d$, then $\ml(G)\ge \floor{1/2 + \sqrt{d}}$.
\end{theorem}
\begin{proof}
Let $H$ be a total ordering of $G$, and let $x_0x_1$ be an edge of maximum height $r$.  Since each row of the height table contains $n$ cells, it follows that $r \ge |E(G)|/n = d/2$.  If $t$ is a positive integer and $\binom{t}{2} < d/2$, then we may apply Lemma~\ref{lem:extlong} to extend $x_0x_1$ to a monotone path of length $t$ in $H$.  Hence, $\binom{t}{2}<d/2$ implies that $\ml(G) \ge t$.  With $t=\floor{1/2 + \sqrt{d}}$, we have that $\binom{t}{2} < d/2$ and therefore $\ml(G) \ge \floor{1/2 + \sqrt{d}}$.
\end{proof}

Let $G$ be a totally ordered graph and let $x_0\ldots x_k$ be a monotone path in $G$.  Viewing height as a resource, extending $x_0 \ldots x_k$ becomes more expensive as $k$ grows.  When extending becomes too expensive, we delete $\{x_0, \ldots, x_{k-2}\}$ from $G$ to form a new totally ordered graph $G'$ (which inherits the orderings of $V(G)$ and $E(G)$), and we extend $x_{k-1}x_k$ to a monotone path in $G'$.  For this to work, we must show that the height of $x_{k-1}x_k$ does not decrease too much when we delete $\{x_0, \ldots, x_{k-2}\}$ from $G$.

\newcommand{\drop}{\mathrm{drop}}

\begin{definition}
Let $G$ be a totally ordered graph.  For $S\subseteq V(G)$ and an edge $e$ in $G-S$, we define $\drop(G,S,e)$ to be $h_G(e) - h_{G-S}(e)$.  For $s\le n-2$, let $g(n,s)$ be the maximum of $\drop(G,S,e)$ over all $n$-vertex totally ordered graphs $G$, all sets $S$ of $s$ vertices in $G$, and all edges $e\in E(G-S)$.
\end{definition}

Note that $g(n,s)$ is monotonic in $n$, since adding isolated vertices to a totally ordered graph $G$ and inserting them arbitrarily into the vertex ordering gives a larger totally ordered graph $G'$ such that $\drop(G,S,e) = \drop(G',S,e)$ for all $S\subseteq V(G)$ and $e\in E(G-S)$.

\begin{lemma}\label{lem:extend-del}
Let $G$ be an $n$-vertex totally ordered graph and let $x_0x_1$ be an edge of height $r$.  If $s$ is a positive integer and $s\le n-2$, then $G$ contains a monotone path extending $x_0x_1$ of length at least $sk + 1$, where $k=\floor{(r-1)/(\binom{s+1}{2}+g(n,s))}$.
\end{lemma}
\begin{proof}
By induction on $n$.  If $k=0$, then the lemma is clear.  Otherwise, $r-1 \ge \binom{s+1}{2} + g(n,s)$ and we may apply Lemma~\ref{lem:extlong} to obtain a monotone path $x_0\ldots x_{s+1}$ of height at least $r-\binom{s+1}{2}$.  Let $S=\{x_0, \ldots, x_{s-1}\}$ and let $G' = G-S$.  We have that $h_{G'}(x_sx_{s+1}) = h_G(x_sx_{s+1}) - \drop(G,S,x_sx_{s+1}) \ge r-\binom{s+1}{2}-g(n,s)$.

Applying the inductive hypothesis to $G'$ and $x_sx_{s+1}$, we obtain a monotone path $P'$ in $G'$ extending $x_sx_{s+1}$ of length at least $sk'+1$, where 
\[ k' = \bfloor{\frac{r-\binom{s+1}{2}-g(n,s)-1}{\binom{s+1}{2}+g(n-s,s)}} \ge \bfloor{\frac{r-\binom{s+1}{2}-g(n,s)-1}{\binom{s+1}{2}+g(n,s)}} = k-1.\]  
Prepending $x_0\ldots x_s$ to $P'$ produces a monotone path in $G$ of length at least $s+sk'+1$, and $s+sk'+1 \ge sk+1$.
\end{proof}

\section{The Token Game}

\newcommand{\dom}{Z}
\newcommand{\domd}{\{1,2,3,\ldots\} \times V(G')}
Our goal is to prove an upper bound on $g(n,s)$.  Let $G$ be an $n$-vertex totally ordered graph, and let $S$ be a set of $s$ vertices of $G$.  We analyze an iterative process which obtains the height table of $G-S$ from the height table of $G$.  Let $G'=G-S$, let $A$ be the array obtained from the height table of $G$ by deleting columns indexed by vertices in $S$, and let $A'$ be the height table of $G'$.  Note that the cells of both $A$ and $A'$ are indexed by $\dom$, where $\dom=\domd$.  We order $\dom$ in the same order as the corresponding cells in $A'$ are defined; that is, $(i,u) \le (i',v)$ if and only if $i<i'$ or $i=i'$ and $u\le v$ in $T(G')$.  For $\beta\in\dom$, the \emph{open down-set} of $\beta$, denoted $D(\beta)$ is $\{\gamma\in\dom\st \gamma < \beta\}$ and the \emph{closed up-set} of $\beta$, denoted $U[\beta]$ is $\{\gamma\in\dom\st \gamma \ge \beta\}$.  Similarly, the \emph{interval} $[\beta,\gamma]$ is $\{\delta\in\dom\st \beta\le\delta\le\gamma\}$.

We produce a sequence of arrays $\{A_\beta\st \beta\in\dom\}$ which initially resemble $A$ and later resemble $A'$.  For $\beta\in\dom$, the cells of $A_\beta$ are indexed by $\dom$ and are partitioned into a \emph{lower part} indexed by $D(\beta)$ and an \emph{upper part} indexed by $U[\beta]$.

For $\beta\in\dom$, each cell in $A_\beta$ is either empty, contains an edge in $G'$, or contains an object called a \emph{hole}.  Moreover, each edge in $G'$ appears in one cell in $A_\beta$.  Each $A_\beta$ also has a \emph{critical interval} $[(i,u),(j,u)]$, where $\beta = (i,u)$ and $j$ is the least integer such that $j\ge i$ and $A_\beta(j,u)$ does not contain a hole.

\begin{lemma}\label{lem:stdtbl}
There is a sequence of arrays $\{A_\beta\st \beta\in\dom\}$ such that each column in the initial array has at most $s$ holes, and for each $\beta\in\dom$ the following hold.
\begin{enumerate}
\item If $\delta < \beta$, then $A_\beta(\delta)=A'(\delta)$.
\item If $\delta \ge \beta$ and $A_\beta(\delta)$ does not contain a hole, then $A_\beta(\delta) = A(\delta)$.
\item If $\gamma$ is the successor of $\beta$ in $\dom$, then $A_\gamma$ is obtained from $A_\beta$ by swapping $A_\beta(\beta)$ and $A_\beta(\delta)$, where $\delta$ is in the critical interval of $A_\beta$.  Moreover, if $\beta$ and $\delta$ index cells in distinct columns $u$ and $v$, then $A_\beta(\delta)=uv$.
\end{enumerate}
\end{lemma}
\begin{proof}
Recall that $A$ is obtained from the height table of $G$ by deleting columns indexed by vertices in $S$.  Note that $A$ omits every edge with both endpoints in $S$ and contains every edge in $G'$.  An edge $uv \in [S,\overline{S}]$ with $u\not\in S$ and $v\in S$ appears in $A$ if and only if $uv$ is in column $u$ in the height table of $G$.  Let $\alpha$ be the minimum element in $\dom$, and let $A_\alpha$ be the array obtained from $A$ by replacing edges in $[S,\overline{S}]$ with holes.  If $u$ indexes a column in $A_\alpha$, then each hole in column $u$ replaces an edge $uv$ in $G$ with $v\in S$, and therefore each column in $A_\alpha$ contains at most $s$ holes.  Clearly, every edge in $G'$ appears once in $A_\alpha$ and $A_\alpha$ satisfies properties (1) and (2).  

We obtain other arrays iteratively.  Let $\beta = (i,u)$, let $\gamma$ be the successor of $\beta$, and suppose that $A_\beta$ has been previously defined but $A_\gamma$ is not yet defined.  Analogously to $A_\beta$, we partition of the cells of $A'$ into a lower part indexed by $D(\beta)$ and an upper part indexed by $U[\beta]$.  Since $A_\beta$ and $A'$ contain the same set of edges and agree on their lower parts, it follows that the upper parts of $A_\beta$ and $A'$ contain the same edges (possibly in a different order).  We consider two cases, depending on whether $A'(\beta)$ is empty or contains an edge in $G'$.

\emph{Case 1}: $A'(\beta)$ is not empty.  Let $e=A'(\beta)$, and let $\delta$ be the index of the cell in $A_\beta$ containing $e$.  We claim that $\delta$ is in the critical interval $[(i,u),(j,u)]$ of $A_\beta$.  Since $e$ is in the upper part of $A'$, it follows that $e$ is in the upper part of $A_\beta$ and so $\delta\ge \beta = (i,u)$.  Since $\delta,(j,u) \in U[\beta]$ and neither $A_\beta(\delta)$ nor $A_\beta(j,u)$ contains a hole, it follows from (2) that $A(\delta) = A_\beta(\delta) = e$ and $A(j,u) = A_\beta(j,u)$.  Suppose for a contradiction that $\delta > (j,u)$.  Note that $e$ is available for $A(j,u)$ when building the height table of $G$, and so $A(j,u) = e'$ for some edge $e'$ incident to $u$ such that $e'>e$ in $T'(G)$.  Since $A_\beta(j,u) = A(j,u) = e'$, it follows that both $e$ and $e'$ appear in the upper part of $A_\beta$ and hence in the upper part of $A'$ also.  Therefore both $e$ and $e'$ are available for $A'(\beta)$ when building the height table of $G'$.  The selection of $e$ over $e'$ for $A'(\beta)$ implies that $e>e'$ in $T'(G')$, contradicting that $e'>e$ in $T'(G)$.  Therefore $\delta \le (j,u)$ and $\delta$ is in the critical interval of $A_\beta$ as claimed.  Obtain $A_\gamma$ from $A_\beta$ by swapping the contents of cells $A_\beta(\beta)$ and $A_\beta(\delta)$ (if $\beta=\delta$, then $A_\gamma = A_\beta$).  Note that if $\delta$ indexes a cell in column $v$ and $v\ne u$, then $A'(\beta) = e$ and $A(\delta) = e$ imply that $e$ is incident to both $u$ and $v$, so that $A_\beta(\delta) = e = uv$, satisfying (3).

We check that $A_\gamma$ satisfies (1) and (2).  Since $\gamma$ is the successor of $\beta$ and $A_\gamma(\beta) = A_\beta(\delta) = e = A'(\beta)$, it follows that $A_\gamma$ satisfies (1).  If the critical interval $[(i,u), (j,u)]$ of $A_\beta$ has size $1$, then $\beta = (i,u) = \delta = (j,u)$ and $A_\gamma = A_\beta$, implying that $A_\gamma$ satisfies (2).  Otherwise $j>i$ and $A_\beta(\beta)$ contains a hole.  Relative to $A_\beta$, the only change in the upper part of $A_\gamma$ is that $A_\gamma(\delta)$ becomes a hole after swapping $A_\beta(\beta)$ and $A_\beta(\delta)$, and so $A_\gamma$ satisfies (2).

\emph{Case 2}: $A'(\beta)$ is empty.  This implies that the upper part of $A'$ contains no edge incident to $u$, and so the upper part of $A_\beta$ also contains no edge incident to $u$.  In particular, $A_\beta(j,u)$ is empty, where $[(i,u), (j,u)]$ is the critical interval of $A_\beta$.  We obtain $A_\gamma$ from $A_\beta$ by swapping the contents of cells $A_\beta(i,u)$ and $A_\beta(j,u)$, satisfying (3).  Since $A_\gamma(\beta)$ and $A'(\beta)$ are both empty, $A_\gamma$ satisfies (1).  Relative to $A_\beta$, the upper part of $A_\gamma$ is either unchanged or contains a new hole at $A_\gamma(j,u)$.  It follows that $A_\gamma$ also satisfies (2). 
\end{proof}

Given the sequence of arrays $\{A_\beta\st\beta\in\dom\}$ from Lemma~\ref{lem:stdtbl}, we obtain a useful upper bound on $\drop(G,S,e)$.

\begin{lemma}\label{lem:drop}
Let $e$ be an edge in $G'$ and choose $\beta\in\dom$ so that $A'(\beta) = e$. If $[(i,u),(j,u)]$ is the critical interval of $A_\beta$, then $\drop(G,S,e) \le j-i$.
\end{lemma}
\begin{proof}
Since $\beta = (i,u)$ and $e$ appears in row $i$ of the height table of $G'$, it follows that $h_{G'}(e) = i$.  Let $\delta$ index the cell in $A_\beta$ containing $e$.  Since the successor $A_\gamma$ of $A_\beta$ satisfies $A_\gamma(\beta) = A'(\beta) = e$, it follows that $A_\gamma$ is obtained from $A_\beta$ by swapping $A_\beta(\beta)$ with $A_\beta(\delta)$.  By (3), we have that $\delta$ is in the critical interval $[(i,u),(j,u)]$ of $A_\beta$, and so $\delta = (\ell,v)$ where $i\le \ell\le j$.  Since $\delta \ge \beta$ and $A_\beta(\delta)$ is not a hole, by (2) we have that $e = A_\beta(\delta) = A(\delta)$.  Therefore $e$ appears in row $\ell$ of the height table of $G$ and so $h_G(e) = \ell$.  We conclude $\drop(G,S,e) = \ell-i \le j-i$.
\end{proof}

We define the \emph{height} of a critical interval $[(i,u),(j,u)]$ to be $j-i$.  Note that the height of the critical interval of $A_\beta$ is at most the number of holes in column $u$ of $A_\beta$.  Also, by property (1) of Lemma~\ref{lem:stdtbl}, all holes of $A_\beta$ are contained in the upper part of $A_\beta$.  Analyzing the movement of the holes as $\beta$ increases in $\dom$ naturally leads to a single player game.

A \emph{token game} is a game played on an array $B$ with rows indexed by the positive integers and columns indexed by a finite list.  Let $B(i,u)$ denote the cell in row $i$ and column $u$.  Each cell in $B$ is empty or contains a \emph{token}.  A token in cell $B(i,u)$ is \emph{grounded} if all cells in column $u$ below $B(i,u)$ contain tokens; a token which is not grounded is \emph{ungrounded}.  One of the columns is distinguished as the \emph{active column}.

A step in a token game modifies $B$ to produce a new array $B'$, subject to certain rules.  Let $u$ be the active column.  If column $u$ contains grounded tokens, then the player may optionally move the highest grounded token in column $u$ from its cell $B(i,u)$ to an empty cell $B(i',v)$, provided that $i'\le i$ and no prior step in the game moved a token between columns $u$ and $v$.  Next, all ungrounded tokens in column $u$ shift down by one cell, and the active column advances cyclically.  A step in which a token moves between columns is a \emph{transfer step}.  The list of arrays produced in a token game is its \emph{transcript}.

\newcommand{\tg}{\hat g}
An \emph{$(n,s)$-token game} is a token game with $n$ columns, each of which initially contains at most $s$ tokens.  Let $\tg(n,s)$ be the maximum number of tokens that can be placed in a single column in an $(n,s)$-token game.  The following gives the connection between $g(n,s)$ and $\tg(n,s)$.  

\begin{lemma}\label{lem:game}
$g(n,s) \le \tg(n-s,s)$
\end{lemma}
\begin{proof}
Let $G$ be an $n$-vertex totally ordered graph and let $S$ be a set of $s$ vertices in $G$ such that $\drop(G,S,e) = g(n,s)$ for some edge $e$ in $G-S$.  Let $G'=G-S$, let $A'$ be the height table of $G'$, obtain $A$ from the height table of $G$ by deleting columns indexed by $S$, and apply Lemma~\ref{lem:stdtbl} to obtain the sequence of arrays $\{A_\beta\st\beta\in\dom\}$.  We use this sequence to play the $(n-s,s)$-token game so that at least $g(n,s)$ tokens are placed in some column.

Construct a sequence $\{B_\beta\st\beta\in\dom\}$ of token arrays as follows.  Let $\beta=(i,u)$.  We put a token in $B_\beta(j,v)$ if and only if $A_\beta(k,v)$ contains a hole, where $k=j+i$ if $v<u$ in $T(G')$ and $k=j+i-1$ otherwise.  Equivalently, we obtain $B_\beta$ from $A_\beta$ by removing all edges so that only holes and empty cells remain, shifting cells down to discard the lower part of $A_\beta$, and replacing holes with tokens.

We claim that the sequence $\{B_\beta\st \beta\in\dom\}$ is the transcript of an $(n-s,s)$-token game in which the active column of $B_\beta$ is the second coordinate in $\beta$.  Let $\alpha$ be the minimum element in $\dom$, and note that each column in $A_\alpha$ contains at most $s$ holes by Lemma~\ref{lem:stdtbl}.  It follows that each column in $B_\alpha$ contains at most $s$ tokens, satisfying the initial condition of an $(n-s,s)$-token game.  

Let $\beta = (i,u)$ and let $\gamma$ be the successor of $\beta$.  From property (3) of Lemma~\ref{lem:stdtbl}, we have that $A_\gamma$ is obtained from $A_\beta$ by swapping $A_\beta(\beta)$ and $A_\beta(\delta)$ for some $\delta$ in the critical interval $[(i,u),(j,u)]$ of $A_\beta$.  If the critical interval has size $1$, then $A_\gamma = A_\beta$ and column $u$ of $B_\beta$ contains no grounded tokens.  We obtain $B_\gamma$ from $B_\beta$ by allowing the tokens in column $u$ to shift down by 1 cell.  The active column advances, completing a legal move in the token game.

Otherwise $j>i$.  Recall that the cells of $B_\beta$ correspond to the upper part of $A_\beta$.  The cells indexed by the critical interval $[(i,u),(j,u)]$ of $A_\beta$ correspond to the cells in $B_\beta$ of height at most $j-i$, except that the last cell $A_\beta(j,u)$ corresponds to $B_\beta(j-i+1,u)$ which has height $j-i+1$.

Since $A_\beta(\ell,u)$ contains a hole for $i\le \ell < j$, it follows that $B_\beta(\ell,u)$ contains a grounded token for $1\le \ell \le j-i$.  Since $A_\beta(\beta)$ contains a hole and $A_\beta(\delta)$ does not, it follows that $\delta > \beta$ and we obtain $A_\gamma$ from $A_\beta$ by swapping the contents of distinct cells $A_\beta(\beta)$ and $A_\beta(\delta)$.  Therefore we obtain $B_\gamma$ from $B_\beta$ by firstly moving the grounded token in $B_\beta(1,u)$ to an empty cell of height at most $j-i$ or to $B_\beta(j-i+1,u)$ and secondly shifting the contents of all cells in column $u$ down by 1 cell.  Equivalently, we obtain $B_\gamma$ from $B_\beta$ by optionally moving the highest grounded token from $B_\beta(j-i, u)$ to an empty cell of height at most $j-i$ and shifting the ungrounded tokens in column $u$ down by 1 cell.  This is allowed in a token game provided that we have not executed a transfer step between a pair of columns more than once.

Suppose that the transition from $B_\beta$ to $B_\gamma$ represents the first transfer step between distinct columns $u$ and $v$; we may assume without loss of generality that a token is moved from column $u$ in $B_\beta$ to column $v$ in $B_\gamma$.  It follows that a hole in $A_\beta(\beta)$ is swapped with the contents of $A_\beta(\delta)$ to form $A_\gamma$, where $\beta$ and $\delta$ index cells in columns $u$ and $v$ respectively.  By property (3) of Lemma~\ref{lem:stdtbl}, we have that $A_\beta(\delta) = uv$.  Since $\delta \ge \beta$, the edge $uv$ is in the upper part of $A_\beta$.  On the other hand, we have $\beta<\gamma$ and $A_\gamma(\beta) = A_\beta(\delta)=uv$, and so $uv$ is in the lower part of $A_\gamma$.  In fact, $A_{\gamma'}(\beta) = A_\gamma(\beta) = uv$ for $\gamma' \ge \gamma$, and so $uv$ is in the lower part of $A_{\gamma'}$ for all $\gamma' \ge \gamma$.  It follows that there are no subsequent transfer steps between columns $u$ and $v$.

Therefore $\{B_\gamma\st\gamma\in\dom\}$ is the sequence of arrays in an $(n-s,s)$-token game.  Let $e$ be an edge in $G'$ with $\drop(G,S,e) = g(n,s)$, and let $\beta$ be the index of the cell in $A'$ containing $e$.  By Lemma~\ref{lem:drop}, we have that $g(n,s) = \drop(G,S,e) \le j-i$, where $[(i,u),(j,u)]$ is the critical interval of $A_\beta$.  Since $A_\beta(\ell,u)$ contains a hole for $i\le \ell < j$, it follows that $B_\beta(\ell, u)$ contains a grounded token for $1\le \ell \le j-i$.  Hence, it is possible to place at least $j-i$ tokens in some column in an $(n-s,s)$-token game and so $\tg(n-s,s) \ge j-i$.
\end{proof}

It remains to analyze the $(n,s)$-token game.  Our main tool is to show that in an $(n,s)$-token game in which the number of tokens in a particular column grows substantially, it is possible to find a subgame with half the number of transfer steps in which a column gains a substantial number of tokens.  Eventually, we obtain a contradiction since the number of tokens in a column cannot grow by more than the total number of transfer steps.

\begin{lemma}\label{lem:tganal}
Let $B_0, \ldots, B_k$ be the transcript of a token game with a total of $m$ tokens and at most $2^\ell$ transfer steps.  Suppose that some column initially contains $a$ tokens in $B_0$ and ends with $b$ tokens in $B_k$.  If $a'$ and $r$ are integers such that $m<(a'+1)r/2$, then some subinterval of $B_0, \ldots, B_k$ is the transcript of a token game with $m$ tokens and at most $2^{\ell-1}$ transfer steps, in which some column initially has at most $a'$ tokens and ends with at least $b'$ tokens, where $b'=b-a-r + 1$.
\end{lemma}
\begin{proof}
Choose $j$ so that both $B_0, \ldots, B_j$ and $B_j, \ldots, B_k$ are transcripts of token games with at most $2^{\ell-1}$ transfer steps.

Let $u$ index a column which initially has $a$ tokens in $B_0$ and ends with $b$ tokens in $B_k$, and let $R$ be the set of tokens which end in column $u$ but were not always in column $u$.  Clearly, $|R| \ge b-a$.  Let $\{t_1, \ldots, t_r\}$ be the tokens in $R$ which are in the $r$ highest positions in $B_k$, and let $R_0 = \{t_1, \ldots, t_r\}$.  Each token in $R_0$ has height at least $b'$ in $B_k$.  Moreover, since the height of a token is non-increasing throughout the game, it follows that each token in $R_0$ has height at least $b'$ in every array.

Since each $t_i \in R_0$ is moved to $u$ during the token game, we may choose columns $v_1, \ldots, v_r$ and indices $\ell_1, \ldots, \ell_r$ such that $t_i$ is moved from $v_i$ in $B_{\ell_i}$ to $u$ in $B_{\ell_i + 1}$.  Since a token game forbids more than one transfer between a pair of columns, $v_1, \ldots, v_r$ are distinct.  Choose $I\in \{[0, j], [j, k]\}$ so that $|R_1| \ge r/2$, where $R_1 = \{t_i \in R_0\st \{\ell_i, \ell_i + 1\} \subseteq I\}$.  Since $t_i$ has height at least $b'$ throughout the game, column $v_i$ in $B_{\ell_i}$ has at least $b'$ grounded tokens.  

Note that it is not possible for each of the columns in $\{v_i\st t_i\in R_1\}$ to begin the subgame $\{B_i\st i\in I\}$ with more than $a'$ tokens, since $(a'+1)|R_1| \ge (a'+1)(r/2) > m$.  It follows that some column $v_i$ has at most $a'$ tokens in the first array of $\{B_i\st i\in I\}$ but has at least $b'$ tokens in $B_{\ell_i}$.
\end{proof}

Iterating Lemma~\ref{lem:tganal} gives the following.

\begin{lemma}\label{lem:tganal2}
In a token game with $m$ tokens and at most $2^\ell$ transfer steps, each column gains a net of at most $1 + 2\ell\ceil{\sqrt{2m}}$ tokens.
\end{lemma}
\begin{proof}
If $\ell = 0$, then the lemma is clear.  For larger $\ell$, suppose that there is a token game $B_0, \ldots, B_k$ with at most $2^\ell$ transfer steps in which some column begins with $a$ tokens and ends with $b$ tokens.  We iterate Lemma~\ref{lem:tganal} with $a'=r=\ceil{\sqrt{2m}}$ to obtain, for each $1\le t\le \ell$, a subgame with $m$ tokens and at most $2^{\ell-t}$ transfer steps in which some column begins with at most $a'$ tokens and ends with at least $(b-a) - (2t-1)r$ tokens.

With $\ell = t$, we obtain a token game with at most $1$ transfer step in which some column begins with at most $a'$ tokens and ends with at least $(b-a)-(2\ell - 1)r$ tokens.  We conclude $(b-a) - (2\ell-1)r - a' \le 1$, which implies $b-a \le 1 + 2\ell r$.
\end{proof}

\begin{corollary}\label{cor:tgbound}
Always $\tg(n,s) \le 4\lg n (\sqrt{2ns} + 1) + s + 1 = O(s+\sqrt{ns}\lg n)$.  In particular, if $n \ge \max\{2,s\}$, then $\tg(n,s) \le 11\sqrt{ns}\lg n$.
\end{corollary}
\begin{proof}
Consider an $(n,s)$-token game in which some column starts with at most $s$ tokens and ends with $\tg(n,s)$ tokens.  Let $b=\tg(n,s)$.  Note that an $(n,s)$-token game contains at most $2^\ell$ transfer steps provided that $2^\ell \ge \binom{n}{2}$; it suffices to choose $\ell=\floor{2\lg n}$.  Let $m$ be the number of tokens in our $(n,s)$-token game; clearly $m\le ns$.  It now follows from Lemma~\ref{lem:tganal2} that $b-s \le 1 + 2\ell\ceil{\sqrt{2m}} \le 1 + 4\lg n (\sqrt{2ns} + 1)$.  When $n\ge \max\{2,s\}$, algebra gives the simpler bound $\tg(n,s) \le 11\sqrt{ns}\lg n$.
\end{proof}

Improvements to Corollary~\ref{cor:tgbound} directly translate to improved bounds on $\ml(G)$ via Lemma~\ref{lem:extend-del} and Lemma~\ref{lem:game}.  Unfortunately, our next theorem shows that there is not much room to improve Corollary~\ref{cor:tgbound}.

\begin{theorem}\label{thm:tglower}
Always $\tg(n,s) \ge \max\{s, \sqrt{2ns} - 3s/2\}$.  Consequently, $\tg(n,s) \ge \Omega(s+\sqrt{ns})$.
\end{theorem}
\begin{proof}
Clearly, $\tg(n,s) \ge s$.  Let $k$ be the largest integer such that $s\binom{k+1}{2} \le n$ and note $k\ge \floor{\sqrt{2n/s} - 1/2}$.  Using that $n\ge s(1+2+\cdots + k)$, we let $M_1, \ldots, M_k$ be disjoint sets of columns such that $|M_j| = sj$ for each $j$.  Let $u_{j,1}, \ldots, u_{j,sj}$ be the columns in $M_j$.  For each $j$, we construct a triangular pattern of tokens in $M_j$ so that for $1\le i\le sj$, the column $u_{j,i}$ contains $i$ grounded tokens.  We assume that the initial positions of all tokens are sufficiently high so that they fall into place as needed.

For $M_1$, we initialize the board so that for $1\le i\le s$, the column $u_{1,i}$ starts with $i$ tokens.  For $j\ge 2$, we assume that we have played the token game so that for $1\le i\le s(j-1)$, the column $u_{j-1,i}$ in $M_{j-1}$ contains $i$ grounded tokens.  We use the tokens in $M_{j-1}$ to construct the desired pattern in $M_j$.  We move the highest grounded token from each column in $u_{j-1,1}, \ldots, u_{j-1,s(j-1)}$ to $u_{j,sj}$ in order.  Since $M_{j-1}$ has $s(j-1)$ columns, this puts $s(j-1)$ tokens in $u_{j,sj}$ and leaves a smaller triangular pattern in $M_{j-1}$ where $u_{j-1,i}$ contains $i-1$ grounded tokens.  Next, we move the highest grounded token from each column in $u_{j-1,2}, \ldots, u_{j-1,s(j-1)}$ to $u_{j,sj-1}$ in order; this places $s(j-1)-1$ tokens in $u_{j,sj-1}$.  Iterating this play, we move all tokens in $M_{j-1}$ to $M_j$.  We complete the triangular pattern by allowing $\min\{i,s\}$ tokens whose initial positions were high in column $u_{j,i}$ to fall into place.

Note that we require at most $s$ tokens in each column initially.  Moreover, since $M_1,\ldots,M_k$ are pairwise disjoint, no pair of columns is involved in more than one transfer step.  After all steps, column $u_{k,sk}$ in $M_k$ contains $sk$ tokens, implying that $\tg(n,s) \ge sk \ge s(\sqrt{2n/s} - 3/2)$.
\end{proof}

Although the bounds in Corollary~\ref{cor:tgbound} and Theorem~\ref{thm:tglower} establish the order of growth of $\tg(n,s)$ up to a logarithmic factor, it would still be interesting to obtain the exact order of growth.  If $\tg(n,s) = O(s+\sqrt{ns})$ as we suspect, then the log term in the lower bound in Corollary~\ref{cor:clique} can be removed.  We do not know how sharp the inequality in Lemma~\ref{lem:game} is; there may be room to make more substantial improvements to our upper bound on $g(n,s)$.  

\section{Dense Graphs}

\begin{theorem}\label{thm:dense}
Let $G$ be an $n$-vertex graph, let $s=n^{1/3}(11\lg n)^{2/3}$, and suppose that $n$ is sufficiently large so that $s\le n-2$.  If $G$ has average degree $d$ and $d>2$, then $\ml(G) > \frac{d}{4s}(1-\frac{2}{d})(1-\frac{1}{s})(1-\frac{4s^2}{d-2})$.
\end{theorem}
\begin{proof}
Let $H$ be a total ordering of $G$.  Since $H$ has $nd/2$ edges, it follows that some edge $x_0x_1$ has height at least $d/2$.  With $s' = \floor{s}$, we apply Lemma~\ref{lem:extend-del} to obtain a monotone path of length at least $s'\floor{\frac{d/2 - 1}{\binom{s'+1}{2}+g(n,s')}} + 1$.  Using Lemma~\ref{lem:game}, Corollary~\ref{cor:tgbound}, and monotonicity of $\tg(n,s)$, we have $g(n,s') \le \tg(n-s',s') \le \tg(n,s') \le 11\sqrt{ns'}\lg n$.  We compute
\begin{align*}
s'\bfloor{\frac{d/2 - 1}{\binom{s'+1}{2}+g(n,s')}} + 1 & > s'\bfloor{\frac{{d/2} - 1}{\binom{s'+1}{2}+11\sqrt{ns'}\lg n}} \\
&\ge (s-1)\bfloor{\frac{{d/2}-1}{\binom{s+1}{2} + 11\sqrt{ns}\lg n}}\\
&= (s-1)\bfloor{\frac{{d/2}-1}{\binom{s+1}{2} + s^2}}\\
&\ge (s-1)\left(\frac{d/2 - 1}{2s^2} - 1\right)\\
&= \frac{d}{4s}\left(1-\frac{2}{d}\right)\left(1-\frac{1}{s}\right)\left(1-\frac{4s^2}{d-2}\right)
\end{align*}
\end{proof}

When the average degree $d$ grows faster than $s^2$, Theorem~\ref{thm:dense} improves R\"odl's result $\ml(G) \ge (1-o(1))\sqrt{d}$.  In terms of $n$, an $n$-vertex graph must have average degree at least $Cn^{2/3}(\lg n)^{4/3}$ for some constant $C$ in order for Theorem~\ref{thm:dense} to offer an improvement.

\begin{corollary}\label{cor:clique}
$\ml(K_n) \ge (\frac{1}{20}-o(1))(\frac{n}{\lg n})^{2/3}$
\end{corollary}
\begin{proof}
By Theorem~\ref{thm:dense}, we have $\ml(K_n) \ge \frac{n-1}{4n^{1/3}(11\lg n)^{2/3}}(1-o(1)) \ge \frac{n^{2/3}}{20(\lg n)^{2/3}}(1-o(1))$.
\end{proof}

We make no attempt to optimize the constant $1/20$.  Echoing remarks of Graham and Kleitman~\cite{GK}, we conjecture that our lower bound on $\ml(K_n)$ is not sharp and that the order of growth of $\ml(K_n)$ is closer to linear than to $(n/\log n)^{2/3}$.

\bibliographystyle{abbrv}
\bibliography{monotone-path}

\end{document}